\documentclass[20pt, oneside]{article}   	
\usepackage{geometry}                		
\usepackage{amsmath}
\usepackage{amsthm}	
\usepackage{amssymb}
\usepackage{tikz}
\usepackage[implicit]{hyperref}
\usepackage[bottom]{footmisc}
\usepackage[affil-it]{authblk}
\usepackage[toc,page]{appendix}
\newtheorem{thm}{Theorem}[section]

\newtheorem{prop}[thm]{Proposition}
\newtheorem{note}[thm]{Note}

\newtheorem{cor}[thm]{Corollary}
\newtheorem*{claim*}{Claim}

\newtheorem*{thm*}{Theorem}

\newtheorem{man}{Theorem}

\theoremstyle{definition}
\newtheorem*{rem*}{Remark}
\newtheorem{defi}[thm]{Definition}
\newtheorem{rem}[thm]{Remark}

\newcommand{\F}{\mathbb{F}}
\newcommand{\fl}{\mathfrak{l}}

\newcommand{\fp}{\mathfrak{p}}

\newcommand{\GL}{\mathrm{GL}}
\newcommand{\SL}{\mathrm{SL}}

\newcommand{\xdownarrow}[1]{%
  {\left\downarrow\vbox to #1{}\right.\kern-\nulldelimiterspace}
}

\title{Natural density of rank-$2$ Drinfeld modules with big Galois image}
\author{Chien-Hua Chen \thanks{Electronic address: \texttt{danny30814@ncts.ntu.edu.tw}; ORCID: \texttt{0000-0003-3267-5603} ; Corresponding author}}
\affil{Mathematics Division,\\ National Center for Theoretical Sciences,\\ Taipei, Taiwan}

\begin{document}

\maketitle
\abstract
In this paper, we compute the natural density of rank-$1$ Drinfeld module over $\F_q[T]$ with surjective adelic Galois representation; and the natural density of  rank-$2$ Drinfeld modules over $\F_q[T]$ whose $\fl$-adic Galois image containing the special linear subgroup for finitely many prime ideal $\fl$.

\section {Introduction}

In the groundbreaking work \cite{J10}, Jones investigated the natural density of elliptic curves over $\mathbb{Q}$ with adelic Galois image that are index-$2$ subgroups of $\GL_2(\hat{\mathbb{Z}})$. Zywina \cite{Z10} later extended this result to include elliptic curves over arbitrary number fields. It turns out that the natural density of elliptic curves over a number field with maximal adelic Galois image is equal to one.

Expanding upon this line of research, \cite{LSTX19} generalized the natural density problem to encompass Abelian varieties with ``big monodromy''. The techniques in these results all involved in classification of maximal subgroups in $\GL_2(\mathbb{Z}/\ell)$ ( or ${\rm GSP}_{2n}(\mathbb{Z}/\ell)$ for abelian varieties), Serre's large sieve method, and Messer-W\"ustholz bound for irreducibility of mod-$\ell$ Galois representations.

In the function field analogy, Pink and R\"utsche \cite{PR09} proved the open image theorem for Drinfeld modules of arbitrary rank without complex multiplication. And \cite{Z11} Zywina proved the existence of rank-$2$ Drinfeld modules with surjective adelic Galois representation. The result is later on generalized by the author to Drinfeld modules of rank $3$ case \cite{C22} and prime rank case \cite{Ch21}.  Therefore, it is reasonable to ask the natural density of Drinfeld modules of prime rank with surjective adelic Galois representation. Due to Van der Heiden's result \cite{Hei03} on determinant of adelic Galois representation for Drinfeld modules, the natural density estimation can be split into two cases:
\begin{enumerate}
\item[\bf Case 1:] Natural density of rank-$1$ Drinfeld modules with surjective adelic Galois representation.

\item[\bf Case 2:] Natural density of rank-$r$ Drinfeld modules whose adelic Galois image containing $\SL_r(\widehat{\F_q[T]})$.

\end{enumerate}
In this paper, we compute and give a brief estimation for case 1. 
\begin{man}[Theorem \ref{main1} \& Corollary \ref{cor1.1}]

For each $x\in\mathbb{R}_{>0}$, let $||\cdot||$ to be the absolute value $|\cdot|_\infty$ of $F_\infty$. Define $$B_1(x)=\{\Delta\in A\mid \Delta\neq0, ||\Delta||<x\}.$$

The upper natural density of rank-$1$ Drinfeld modules with surjective adelic Galois representation is equal to
$$
\begin{array}{lll}
1-\lim_{x\rightarrow \infty}\frac{\#\{\Delta\in B_1(x)\mid {\rm Im}\rho_{\phi^{\Delta}}\neq\hat{A}^*\}}{\#B_1(x)}\\
\ \\
\ \\
=1-\frac{\sum_{2\leqslant t\leqslant q-2, t\mid q-1}\left[\sum_{i=1}^{\lfloor\frac{\lfloor{\rm log}_qx\rfloor-1}{t}\rfloor}\left(q^{ti+1}\cdot\sum_{j\mid i}\mu(j)\lfloor\frac{q-1}{tj}\rfloor\right)+\left(q\cdot1\right)\right]+\sum_{i=1}^{\lfloor\frac{\lfloor{\rm log}_qx\rfloor-1}{q-1}\rfloor}\left(q^{(q-1)i+1}\right)}{\#B_1(x)}\\
\ \\
\ \\
\leqslant1-\sum_{2\leqslant t\leqslant q-2, {t\mid q-1}}\frac{q-1}{(q^t-1)\cdot t}
\end{array}
$$
\end{man}

The computation is based on Gekeler's \cite{Gek16} characterization of adelic Galois image for rank-$1$ Drinfeld modules, which he called twisted Carlitz modules.

On the other hand, for case 2 we only compute the natural density of rank-$2$ Drinfeld modules with $\fl$-adic Galois image containing $\SL_2(\F_q[T]_\fl)$ for finitely many prime ideal $\fl$. 

\begin{man}[Corollary \ref{main2}]
Let $q\geqslant 4$ be a prime power, and $S$ be a finite set of prime ideals of $A=\F_q[T]$. Then we have
\begin{align*}
L_2^{S,low}&:=\liminf_{x\rightarrow \infty}\frac{\#\{(g_1,g_2)\in B_2(x)\mid {\rm Im}\rho_{\phi^{(g_1,g_2)},\fl}\supset{\rm SL}_2(A_\fl) \textrm{ for }\fl\in S\}}{\#B_2(x)}\\
&\geqslant 1+\sum_{i=1}^{|S|}(-1)^i\cdot\left(\sum_{\fl_1,\cdots, \fl_i\in S \textrm{ distinct }}\frac{1}{q^{\deg_T\fl_1+\cdots+\fl_i}} \right).
\end{align*}

\end{man}

The reason why we can only estimate the natural density for finitely many $\fl$-adic Galois representations is because of the followings: 

First, the technical difficulty on the criterion to determine whether a subgroup of $\GL_2(\F_q[T]/\fl)$ contains $\SL_2(\F_q[T]/\fl)$, one can compare Proposition \ref{modl} with the elliptic curve situation (\cite{Z10}, Lemma A.8). Second, there is no function field analogue of Messer-W\"ustholz bound for surjective mod $\fl$ Galois representations (see \cite{Z11}, Theorem 5.8 for the elliptic curve case). Currently, in \cite{C23} we have derived Drinfeld module analogue of Messer-W\"ustholz bound for irreducible mod $\fl$ Galois representations, but there is no such a bound for a Drinfeld module $\phi$ whose mod $\fl$ Galois image lies in a normalizer of Cartan subgroup of $\GL_2(\F_q[T]/\fl)$. The current existed bound needs to assume that $\phi$ has good reduction at $\fl$, see \cite{CL19} Theorem 2.1 (2).

Concluding the introduction, we mention that Ray \cite{R23} formulated and proved a lower bound for natural density of rank-$2$ Drinfeld modules with surjective $(T)$-adic Galois representation. He found directly that a special type of rank-2 Drinfeld module satisfies Pink's criteria (\cite{PR09}, Proposition 4.1). And his idea is motivated by Proposition 4.1 in \cite{Zy11}. However, that type of rank-$2$ Drinfeld module has very small natural density, this leads to his estimation of natural density for Drinfeld modules with surjective $(T)$-adic representation is fairly small (density  $>q^{-7}$).

 Apart from his result, we give a stronger estimation (see Corollary \ref{main3}):
\begin{align*}
\liminf_{x\rightarrow \infty}\frac{\#\{(g_1,g_2)\in B_2(x)\mid {\rm Im}\rho_{\phi^{(g_1,g_2)},(T)}={\rm GL}_2(A_{(T)})\}}{\#B_2(x)}\geqslant (q-1)\cdot (\frac{1}{q}-\sum_{2\leqslant t\leqslant q-2, {t\mid q-1}}\frac{1}{(q^t-1)}).
\end{align*}

\section{Preliminary}

\subsection{Notation}

\begin{itemize}
\item$q=p^e$ is a prime power
\item$A=\mathbb{F}_q[T]$
\item$F=\mathbb{F}_q(T)$
\item$F^{{\rm{sep}}}$= separable closure of $F$
\item$F^{{\rm{alg}}}$= algebraic closure of $F$
\item$G_F= {\rm{Gal}}(F^{{\rm{sep}}}/F)$
\item $\fl=(l)$ a nonzero prime ideal of $A$ with monic generator $l$, and define $\deg_T(\fl)=\deg_T(l)$ 
\item$A_{\mathfrak{p}}=$ completion of A at the nonzero prime ideal $\mathfrak{p}\triangleleft A$ 
\item$\mathbb{F}_{\mathfrak{p}}= A/\mathfrak{p}$ 
\end{itemize}

\subsection{Drinfeld modules and Galois representations}
Let $A=\F_q[T]$ be the polynomial ring over finite field with $q=p^e$ an odd prime power, $F=\F_q(T)$ be the fractional field of $A$, and  $K$ be a finite extension over $F$. Set $F\{\tau\}$ to be the twisted polynomial ring with usual addition rule, and the multiplication rule is defined to be $\tau\alpha=\alpha^q\tau \text{ for any } \alpha\in K$.

\begin{defi}
A Drinfeld $A$-module of rank $r$ over $F$ is a ring homomorphism

$$\phi: A\rightarrow F\{\tau\}$$
such that
\begin{enumerate}
\item[(i)] $\phi(a):=\phi_a$ satisfies ${\rm deg}_{\tau}\phi_a=r\cdot {\rm deg}_Ta$
\item[(ii)] Denote $\partial: F\{\tau\}\rightarrow F$ by $\partial(\sum a_i\tau^i)=a_0$, then $\phi$ satisfies $\gamma=\partial\circ\phi$.
\end{enumerate}
\end{defi}

From the definition of Drinfeld $A$-module, we can characterize a Drinfeld module $\phi$ by writing down $$\phi_T=T+g_1\tau+\cdots+g_{r-1}\tau^{r-1}+g_r\tau^r, \textrm{ where } g_i\in F \textrm{ and }g_r\in F^*.$$

\begin{defi}
Let $\phi$ and $\psi$ be two rank-$r$ Drinfeld $A$-modules over $K$. A {\bf{morphism}} $u:\phi\rightarrow \psi$ over $K$ is a twisted polynomial $u\in K\{\tau\}$ such that
$$u\phi_a=\psi_a u \text{ \rm for all } a\in A.$$ 
A non-zero morphism $u:\phi\rightarrow \psi$ is called an isogeny. A morphism $u:\phi\rightarrow \psi$ is called an {\bf{isomorphism}} if its inverse exists. 
\end{defi}

\begin{defi}
\begin{enumerate}
\item[(i)] Let $\phi$ be a Drinfeld $A$-module over $K$. The Drinfeld module $\phi$ gives $K^{{\rm{alg}}}$ an $A$-module structure, where $a\in A$ acts on $K^{{\rm{alg}}}$ via $\phi_{a}$.

\item[(ii)] Let $\mathfrak{a}$ be an ideal of $A$ with monic generator $a$.The {{$\mathfrak{a}$-torsion}} of $\phi$ is defined to be the set $$\phi[\mathfrak{a}]=\phi[a]:=\{ {\rm{zeros}}\ {\rm{of}}\ \phi_a(x)= \gamma(a)x+ \sum^{r\cdot {\rm{deg}}(a)}_{i=1}g_i(a)x^{q^i} \}\subseteq K^{alg}.$$ The action $b\cdot \alpha=\phi_b(\alpha) \ \forall\ b\in A, \forall \alpha\in\phi[\mathfrak{a}]$ also gives $\phi[\mathfrak{a}]$ an $A$-module structure. 
\end{enumerate}
\end{defi}

\begin{prop}\label{prop0.2}
Let $\phi$ be a rank $r$ Drinfeld module over $K$ and $\mathfrak{a}$ a non-zero ideal of $A$,
, then the $A/\mathfrak{a}$-module $\phi[\mathfrak{a}]$ is free of rank $r$.

\end{prop}
\begin{proof}
See \cite{Goss} Proposition 4.5.7.
\end{proof}

Let $\phi$ be a rank $r$ Drinfeld module over $F$ of generic characteristic, then $\phi_a(x)$ is separable, so we have $\phi[a]\subseteq F^{{\rm{sep}}}$.
This implies that $\phi[a]$ has a $G_F$-module structure. Given a nonzero prime ideal $\fl$ of $A$, we can consider the $G_F$-module $\phi[\mathfrak{l}]$. We obtain the so-called \textbf{mod $\mathfrak{l}$ Galois representation}
$$\bar{\rho}_{\phi,\mathfrak{l}}:G_F\longrightarrow {\rm{Aut}}(\phi[\mathfrak{l}])\cong GL_r(\mathbb{F}_{\mathfrak{l}}).$$
For a prime power $\fl^i$, we may construct the residual representation 
$$\bar{\rho}_{\phi,\mathfrak{l}^i}:G_F\longrightarrow {\rm{Aut}}(\phi[\mathfrak{l}^i])\cong GL_r(A/{\mathfrak{l}}^i).$$ By taking inverse limit with respect to $\mathfrak{l}^i$, we have the \textbf{$\mathfrak{l}$-adic Galois representation}
$${\rho}_{\phi,\mathfrak{l}}: G_F \longrightarrow \varprojlim_{i}{\rm{Aut}}(\phi[\mathfrak{l^i}])\cong {\rm{GL_r}}(A_{\mathfrak{l}}).$$

\subsection{Reduction of Drinfeld modules}

\begin{defi}

Let $K$ be a local field with uniformizer $\pi$, valuation ring $R$, unique maximal ideal $\fp:=(\pi)$, normalized valuation $v$ and residue field $\F_\fp$. Let $\phi: A\longrightarrow K\{\tau\}$ be a Drinfeld module of rank $r$. 

\begin{enumerate}
\item[(i)]
We say that $\phi$ has {stable reduction} if there is a Drinfeld module $\phi':A\longrightarrow R\{\tau\}$ such that 
\begin{enumerate}
\item[1.] $\phi'$ is isomorphic to $\phi$ over $K$;
\item[2.] $\phi'$ mod $\fp$ is still a Drinfeld module (i.e. $\phi'_{T}$ mod $\mathfrak{p}$ has ${\rm{deg}}_{\tau}\geqslant 1$ ).
\end{enumerate}
\item[(ii)]$\phi$ is said to have {stable bad reduction of rank $r_1<r$} if $\phi$ has stable reduction and $\phi$ mod $\mathfrak{p}$ has rank $r_1$.\\
\item[(iii)] $\phi$ is said to have {good reduction} if $\phi$ has stable reduction and $\phi$ mod $\mathfrak{p}$ has rank $r$.

\end{enumerate}

\end{defi}

\begin{rem}\ \\
{\rm{We denote $\phi \mod \mathfrak{p}$ by $\phi \otimes \mathbb{F}_{\mathfrak{p}}$.}}

\end{rem}

\section{Density of rank-$1$ Drinfeld modules}

In this section, we apply the following result from Gekeler to compute the natural density of rank-$1$ Drinfeld modules whose adelic Galois representation is surjective.

\begin{thm}[\cite{Gek16}, Corollary 3.14]\label{rk1case}
Let $\phi^{\Delta}$ be the rank-$1$ Drinfeld $A$-module determined by $$\phi^{\Delta}_T=T+\Delta\tau,$$ where $\Delta\in A-\{0\}$. Then the image of adelic Galois representation of $\phi^{\Delta}$ is a subgroup of $\hat{A}^*$ of index $${\rm def}(\rho_{\phi^{\Delta}})=g.c.d.(d-1, q-1, k_0^*).$$

Here we write $\Delta=c^{k_0}P_1^{e_1}\cdots P_s^{e_s}$ to be the prime factorization of $\Delta$ in $A$, and $c$ is a generator of $\F_q^*$. The notation  $d:=\deg_T\Delta$, and  $$k_0^*=\begin{cases} k_0, & \text{if $q$ or $d=\deg\Delta$ is even}\\
\\
\text{the unique $k \equiv k_0+(q-1)/2 \mod q-1$ with $0\leqslant k< q-1$} , &\text{otherwise}
\end{cases}$$
\end{thm}

Now we define the natural density of rank-$1$ Drinfeld module with surjective adelic representation:

\begin{defi}
\begin{enumerate}
\item For each $x\in\mathbb{R}_{>0}$, let $||\cdot||$ to be the absolute value $|\cdot|_\infty$ of $F_\infty$. Define $$B_1(x)=\{\Delta\in A\mid \Delta\neq0, ||\Delta||<x\}.$$
\item The natural density (if exists) of rank-$1$ Drinfeld module with surjective adelic representation is defined to be $$L_1=\lim_{x\rightarrow \infty}\frac{\#\{\Delta\in B_1(x)\mid {\rm Im}\rho_{\phi^{\Delta}}=\hat{A}^*\}}{\#B_1(x)}.$$ 
\end{enumerate}
\end{defi}

\begin{thm}\label{main1}
The number of $\Delta\in B_1(x)$ with ${\rm Im}\rho_{\phi^{\Delta}}\neq\hat{A}^*$ is equal to

$$\sum_{2\leqslant t\leqslant q-2, t\mid q-1}\left[\sum_{i=1}^{\lfloor\frac{\lfloor{\rm log}_qx\rfloor-1}{t}\rfloor}\left(q^{ti+1}\cdot\sum_{j\mid i}\mu(j)\lfloor\frac{q-1}{tj}\rfloor\right)+\left(q\cdot1\right)\right]+\sum_{i=1}^{\lfloor\frac{\lfloor{\rm log}_qx\rfloor-1}{q-1}\rfloor}\left(q^{(q-1)i+1}\right)$$

\end{thm}
\begin{proof}

From Theorem \ref{rk1case}, those $\Delta\in A$ such that ${\rm Im}\rho_{\phi^\Delta}\neq \hat{A}^*$ have the following property: $\deg_T \Delta-1$, $q-1$ and $k_0^*$ have a nontrivial common divisor. Thus we know that the possible common divisor can only be those $2\leqslant t\leqslant q-1$ with $t\mid q-1$. Now we count the number of $\Delta\in B_1(x)$ with ${\rm def}(\rho_{\phi^{\Delta}})=g.c.d.(d-1, q-1, k_0^*)=t.$ 

 We write $\Delta=c^{k_0}P_1^{e_1}\cdots P_s^{e_s}$. In the first summation, for each $2\leqslant t\leqslant q-2,\textrm{ and } t\mid q-1$, we consider $1\leqslant i\leqslant \lfloor\frac{\lfloor{\rm log}_qx\rfloor-1}{t}\rfloor$. The number $q^{ti+1}$ represents the number of monic $\Delta\in B_1(x)$ whose degree is equal to $ti+1$, and the summation 
$$\sum_{j\mid i}\mu(j)\lfloor\frac{q-1}{tj}\rfloor$$
represents the choices of $c^{k_0}$ such that $k_0^*$ is a multiple of $t$ but $k_0^*/t$ is coprime to $i$. The term $(q\cdot 1)$ comes from the extreme case when $\Delta$ has degree equal to $1$, it forces $k_0^*$ to be equal to $t$.

In the second sum, it presents the case when $t=q-1$, so we must have $k_0^*=0$, and the degree of $\Delta$ must be a multiple of $q-1$.

Hence the result follows

\end{proof}

\begin{rem}
One can check that, as a function of $x$, the summation $$\frac{\sum_{2\leqslant t\leqslant q-2, t\mid q-1}\left[\sum_{i=1}^{\lfloor\frac{\lfloor{\rm log}_qx\rfloor-1}{t}\rfloor}\left(q^{ti+1}\cdot\sum_{j\mid i}\mu(j)\lfloor\frac{q-1}{tj}\rfloor\right)+\left(q\cdot1\right)\right]+\sum_{i=1}^{\lfloor\frac{\lfloor{\rm log}_qx\rfloor-1}{q-1}\rfloor}\left(q^{(q-1)i+1}\right)}{\#B_1(x)}$$ is monotone decreasing and lower bounded by zero. Hence  $\lim_{x\rightarrow \infty}\frac{\#\{\Delta\in B_1(x)\mid {\rm Im}\rho_{\phi^{\Delta}}\neq\hat{A}^*\}}{\#B_1(x)}$ exists, and we get $$L_1=1-\lim_{x\rightarrow \infty}\frac{\#\{\Delta\in B_1(x)\mid {\rm Im}\rho_{\phi^{\Delta}}\neq\hat{A}^*\}}{\#B_1(x)}.$$

\end{rem}
\begin{cor}\label{cor1.1}
$$\lim_{x\rightarrow \infty}\frac{\#\{\Delta\in B_1(x)\mid {\rm Im}\rho_{\phi^{\Delta}}\neq\hat{A}^*\}}{\#B_1(x)}\geqslant\sum_{2\leqslant t\leqslant q-2, \\{t\mid q-1}}\frac{q-1}{(q^t-1)\cdot t}$$
\end{cor}
\begin{proof}
The inequality is obtained by direct computation and the fact that $$q-1\geqslant\sum_{j\mid i}\mu(j)\lfloor\frac{q-1}{tj}\rfloor\geqslant\frac{q-1}{t}.$$

Therefore, on the right inequality we have
$$
\begin{array}{lll}
\lim_{x\rightarrow \infty}\frac{\#\{\Delta\in B_1(x)\mid {\rm Im}\rho_{\phi^{\Delta}}\neq\hat{A}^*\}}{\#B_1(x)}\\
\ \\
\ \\
\geqslant \lim_{x\rightarrow \infty}\frac{\sum_{2\leqslant t\leqslant q-2, t\mid q-1}\sum_{i=1}^{\lfloor\frac{\lfloor{\rm log}_qx\rfloor-1}{t}\rfloor}\left(q^{ti+1}\cdot\sum_{j\mid i}\mu(j)\lfloor\frac{q-1}{tj}\rfloor\right)}{\#B_1(x)}\\
\ \\
\ \\
\geqslant\sum_{2\leqslant t\leqslant q-2, {t\mid q-1}}\frac{q-1}{t}\cdot\left(\lim_{x\rightarrow \infty}\frac{\sum_{i=0}^{\lfloor\frac{\lfloor{\rm log}_qx\rfloor-1}{t}\rfloor}q^{ti+1}}{q^{\lfloor {\rm log}_qx\rfloor}-1}\right).
\end{array}$$

Now we compute the limit

$$\lim_{x\rightarrow \infty}\frac{\sum_{i=0}^{\lfloor\frac{\lfloor{\rm log}_qx\rfloor-1}{t}\rfloor}q^{ti+1}}{q^{\lfloor x\rfloor}-1}=\lim_{x\rightarrow \infty}\frac{q\cdot (q^{t\cdot\lfloor\frac{\lfloor{\rm log}_qx\rfloor-1}{t}\rfloor}-1)}{(q^t-1)(q^{\lfloor{\rm log}_q x\rfloor}-1)}=\frac{1}{q^t-1},$$
hence the inequality in the statement follows.

\end{proof}

\begin{rem}
It is worth to mention that apart from the natural density of elliptic curves with maximal adelic Galois representation, the above corollary shows that $L_1<1$. 

This implies that NOT almost all rank-$1$ Drinfeld modules over $A$ have maximal adelic Galois representation. Combining with the following proposition, we can say that NOT almost all rank-$r$ Drinfeld modules over $A$ has surjective adelic Galois representation.

\end{rem}

\begin{prop}[\cite{Hei03}, Proposition 7.1] \label{weilpair}
Let $\phi$ be a Drinfeld $A$-module over $F$ of rank $r$ defined by $\phi_T=T+g_1\tau+g_2\tau^2+\cdots+g_r\tau^r,\ \text{where $g_r\in F^*$.}$
Then $\det\circ\bar{\rho}_{\phi,\mathfrak{a}}=\bar{\rho}_{\psi,\mathfrak{a}}$ where $\psi$ is the Drinfeld module over $F$ of rank $1$ defined by $$\psi_T=T+(-1)^{r-1}g_r\tau.$$
\end{prop}

\section{Density of rank-$2$ Drinfeld modules}
From now on, we make a further assumption that
\begin{enumerate}
\item[$\bullet$] $q\geqslant 4$
\end{enumerate}
In this section, let $S$ be a finite set of prime ideals in $A$. We aim to estimate the natural density of rank-$2$ Drinfeld modules whose $\fl$-adic Galois representation is large enough. 

\begin{defi}
\begin{enumerate}
\item For each $x\in\mathbb{R}_{>0}$, let $||\cdot||$ to be the absolute value $|\cdot|_\infty$ of $F_\infty$. Define $$B_2(x)=\{(g_1,g_2)\in A^2\mid g_2\neq0, ||g_i||<x\}.$$
\item Similar to the rank-$1$ case, the natural density (if exists) of rank-$2$ Drinfeld module with large finite $\fl$-adic representation is defined to be $$L_2^S=\lim_{x\rightarrow \infty}\frac{\#\{(g_1,g_2)\in B_2(x)\mid {\rm Im}\rho_{\phi^{(g_1,g_2)},\fl}\supset{\rm SL}_2(A_\fl) \textrm{ for }\fl\in S\}}{\#B_2(x)}.$$

We  define the upper (resp. lower) density as following:

$$L_2^{S,up}=\limsup_{x\rightarrow \infty}\frac{\#\{(g_1,g_2)\in B_2(x)\mid {\rm Im}\rho_{\phi^{(g_1,g_2)},\fl}\supset{\rm SL}_2(A_\fl) \textrm{ for }\fl\in S\}}{\#B_2(x)};$$

$$L_2^{S,low}=\liminf_{x\rightarrow \infty}\frac{\#\{(g_1,g_2)\in B_2(x)\mid {\rm Im}\rho_{\phi^{(g_1,g_2)},\fl}\supset{\rm SL}_2(A_\fl) \textrm{ for }\fl\in S\}}{\#B_2(x)}.$$

\end{enumerate}
\end{defi}

\subsection{Some group theory}
\begin{prop}\label{modl}
Let $\fl$ be a prime ideal of $A$, and $H\subset \GL_2(\F_\fl)$. If $H\cap C\neq \emptyset$ for any conjugacy class $C\subset \SL_2(F_\fl)$, then $H\supset \SL_2(F_\fl).$

\end{prop}
\begin{proof}

Consider the subgroup $H\cap \SL_2(\F_\fl)\subset \SL_2(\F_\fl)$. Since $H\cap \SL_2(\F_\fl)$ intersect with any conjugacy class of $\SL_2(\F_\fl)$, it is well known that $H\cap \SL_2(\F_\fl)=\SL_2(\F_\fl)$.
\end{proof}

\begin{defi}

Denote $\mathcal{C}$ to be the collection of conjugacy classes $C$ in $\SL_2(\F_\fl)$

\end{defi}
The following result gives a criteria so that we can deduce from mod-$\fl$ representation toward $\fl$-adic representation. Here we restrict ourselves to the rank-$2$ case.

\begin{thm}[\cite{PR09}]
Let $\phi$ be a rank-$2$ Drinfeld module over $A$. Suppose that
\begin{enumerate}
\item ${\rm Im}\bar{\rho}_{\phi,\fl}\supset \SL_2(A/\fl)$,

\item ${\rm Im}\bar{\rho}_{\phi,\fl^2}$ contains a non-scalar matrix that becomes trivial after modulo $\fl$
,\end{enumerate}
then the $\fl$-adic Galois representation of $\phi$ satisfies
$${\rm Im}\rho_{\phi,\fl}\supset\SL_2(A_\fl).$$
\end{thm}
\begin{note}
The assumption $q\geqslant 4$ is required for the Theorem.
\end{note}
\begin{proof}
See Proposition 4.1 in \cite{PR09}. Note that the statements in Proposition 4.1 needs a condition on $\det\circ{\rm Im}\bar{\rho}_{\phi,\fl}$ to make the image of $\fl$-adic representation become $\GL_2(A_\fl)$. At here we only focus on the relationship between ${\rm Im}\bar{\rho}_{\phi,\fl}$ and $\SL_2(A_\fl)$, hence we drop the requirement on determinant.
\end{proof}

Now we fix a prime $\fl\in S$, and compute the natural density

$$\widehat{L}_{2,\fl}=\lim_{x\rightarrow \infty}\frac{\#\{(g_1,g_2)\in B_2(x)\mid {\rm Im}\bar{\rho}_{\phi^{(g_1,g_2)},\fl}\not\supset{\rm SL}_2(A/\fl) \textrm{ for }\fl\in S\}}{\#B_2(x)}.$$

In section the following section, we will show that $\widehat{L}_{2,\fl}$ exists and is equal to $0$.
From Proposition \ref{modl}, we have 
$$
\widehat{L}_{2,\fl}\leqslant \lim_{x\rightarrow \infty}\sum_{C\in\mathcal{C}}\frac{\#\{(g_1,g_2)\in B_2(x)\mid {\rm Im}\bar{\rho}_{\phi^{(g_1,g_2)}}\cap C=\emptyset\}}{\#B_2(x)}$$

\subsection{Large sieve inequality}

In this subsection, we apply function field analogue of large sieve method to estimate the number $\#Y_{C}(x)=\#\{(g_1,g_2)\in B_2(x)\mid {\rm Im}\bar{\rho}_{\phi^{(g_1,g_2)}}\cap C=\emptyset\}$ when $x$ is large.

\begin{thm}[\cite{H96}, Theorem 3.2]\label{lsm}

Let $K$ be a positive integer. For each monic irreducible polynomial $\mathbf{p}\in A$ with $\deg \mathbf{p}\geqslant 1$, let $a_\mathbf{p}$ be a real number such that $0<a_\mathbf{p}\leqslant 1$. Set $X\subset A^2$ to be a subset satisfies the following property:
$$\# X_\mathbf{p}\leqslant a_\mathbf{p}\cdot q^{2\deg_T\mathbf{p}}=a_{\mathbf{p}}\cdot\#({A^2/\mathbf{p}\cdot A^2}),$$
where $X_\mathbf{p}\subset A^2/\mathbf{p}A^2$ denotes the canonical image of $X$ in the quotient group $A^2/\mathbf{p}A^2$.

Then $$\#(X\cap B_2(x))\cdot L(K)\leqslant q^{2(m_0+1)},$$ 

where 
$$
\begin{array}{lll}
L(K)=1+\sum_{M\in MS, 1\leqslant\deg_T M\leqslant K}\prod_{\mathbf{p}\in MI, \mathbf{p}\mid M}\left(\frac{1-a_\mathbf{p}}{a_\mathbf{p}}\right),\\
\ \\
MS= \textrm{ all monic square-free polynomials in }\F_q[T],\\
\ \\
MI= \textrm{ all monic irreducible polynomials with positive degree in} A,\\
\ \\
m_0={\rm sup}\{{\rm log}\ x, 2K-1\}
\end{array}
$$
\end{thm}

In order to apply the large sieve method, set $K=\frac{1}{2}{\rm log} x$ and $X\cap B_2(x)=Y_C(x)$. Let $\sum{(\frac{1}{2}{\rm log} x)}=\{\mathbf{p}\neq \fl \textrm{ prime ideal of }A\mid \deg_T\mathbf{p}\leqslant \frac{1}{2}{\rm log} x\}$. For $\mathbf{p}\not\in \sum(\frac{1}{2}{\rm log} x)$, define $\alpha_{\mathbf{p}}=0$. And for $\mathbf{p}\in \sum(\frac{1}{2}{\rm log} x)$, define
$$
\begin{array}{lll}
\Omega_{\mathbf{p}, C}=\{(\bar{g}_1,\bar{g}_2)\in \F_\mathbf{p}^2\mid \bar{g}_2\neq 0, \bar{\rho}_{\phi^{(\bar{g}_1,\bar{g}_2)},\fl}({\rm Frob}_\mathbf{p})\in C\}\\
\ \\
\alpha_{\mathbf{p},C}=\frac{\#\Omega_{\mathbf{p},C}}{|\F_\mathbf{p}|^2}
\end{array}
.$$

 Now we examine that $Y_{C}(x)$ satisfies the assumption of the above theorem. Let $Y_{{\bf{p}},C}$ be its canonical image  modulo $\mathbf{p}$, then
$$Y_{{\bf{p}},C}\subset \F_\mathbf{p}^2-\Omega_{\mathbf{p},C}\Rightarrow \#Y_{{\bf{p}},C}\leqslant (1-\alpha_{\mathbf{p},C})|\F_\mathbf{p}|^2.$$

Therefore, Theorem \ref{lsm} implies

$$Y_{C}(x)\cdot L(\frac{1}{2}{\rm log} x)\leqslant q^{2({\rm log} x+1)},$$

and $L(\frac{1}{2}{\rm log} x)=1+\sum_{M\in MS, 1\leqslant\deg_T M\leqslant \frac{1}{2}{\rm log} x}\prod_{\mathbf{p}\in MI, \mathbf{p}\mid M}\left(\frac{\alpha_{\mathbf{p},C}}{1-\alpha_{\mathbf{p},C}}\right)$.

Note that $$L(\frac{1}{2}{\rm log} x)\geqslant \sum_{\mathbf{p}\in \sum(\frac{1}{2}{\rm log} x)}\frac{\alpha_{\mathbf{p},C}}{1-\alpha_{\mathbf{p},C}}\geqslant \sum_{\mathbf{p}\in \sum(\frac{1}{2}{\rm log} x)} \alpha_{\mathbf{p},C}.$$

The next step is to estimate the number $$\alpha_{\mathbf{p},C}=\frac{\#\Omega_{\mathbf{p},C}}{|\F_\mathbf{p}|^2}.$$

By Corollary of Proposition 7 in \cite{Yu95}, we have the following property:

\begin{prop} 
The number of isomorphism classes of rank-$2$ Drinfeld module $\phi\otimes \F_\mathbf{p}$ such that $ \bar{\rho}_{\phi\otimes \F_\mathbf{p},\fl}({\rm Frob}_\mathbf{p})\in C$  is equal to $h(\mathcal{O})=\#{\rm Pic}(\mathcal{O})$, where $\mathcal{O}$ is an imaginary quadratic order over $A$ depending only on $C$.

\end{prop}

As a corollary, we immediately have
$$\#\Omega_{\mathbf{p},C}=\frac{|\F_\mathbf{p}^*|}{|\mathcal{O}^*|}h(\mathcal{O})$$

\begin{thm}
$$\lim_{x\rightarrow \infty}\frac{\sum_{C\in\mathcal{C}}Y_{C}(x)}{B_2(x)}=0.$$

\end{thm}

\begin{proof}
Since $L(\frac{1}{2}{\rm log} x)\geqslant \sum_{\mathbf{p}\in \sum(\frac{1}{2}{\rm log} x)} \alpha_\mathbf{p}$ and $\#\Omega_{\mathbf{p},C_i}=\frac{|\F_\mathbf{p}^*|}{|\mathcal{O}^*|}h(\mathcal{O})$, we have 
$$L(\frac{1}{2}{\rm log} x)\geqslant \frac{h(\mathcal{O})}{|\mathcal{O^*}|}\cdot \sum_{\mathbf{p}\in \sum(\frac{1}{2}{\rm log} x)}\frac{|\F_\mathbf{p}^*|}{|\F_\mathbf{p}|^2}\geqslant\frac{h(\mathcal{O})}{2|\mathcal{O^*}|}\cdot \sum_{\mathbf{p}\in \sum(\frac{1}{2}{\rm log} x)}\frac{1}{|\F_\mathbf{p}|}.$$

Therefore, we get
$$\frac{1}{L(\frac{1}{2}{\rm log} x)}\leqslant\frac{2|\mathcal{O}^*|}{h(\mathcal{O})}\cdot\frac{1}{\sum_{\mathbf{p}\in \sum(\frac{1}{2}{\rm log} x)}\frac{1}{|\F_\mathbf{p}|}}.$$
This implies 
$$Y_C(x)\leqslant\frac{x^2q^2\cdot2|\mathcal{O}^*|}{h(\mathcal{O})}\cdot\frac{1}{\sum_{\mathbf{p}\in \sum(\frac{1}{2}{\rm log} x)}\frac{1}{|\F_\mathbf{p}|}}.$$

On the other hand, $\#B_2(x)=(q^{{\rm log} x})^2=x^2$. Thus
$$\lim_{x\rightarrow \infty}\frac{\sum_{C\in \mathcal{C}}Y_C(x)}{B_2(x)}\leqslant\sum_{C\in\mathcal{C}}\lim_{x\rightarrow \infty}\frac{q^2\cdot2|\mathcal{O}^*|}{h(\mathcal{O})}\cdot\frac{1}{\sum_{\mathbf{p}\in \sum(\frac{1}{2}{\rm log} x)}\frac{1}{|\F_\mathbf{p}|}}=\frac{q^2\cdot2|\mathcal{O}^*|}{h(\mathcal{O})}\cdot\frac{1}{\sum_{\mathbf{p}\neq\fl \textrm{ prime of }A}\frac{1}{|\F_\mathbf{p}|}}=0$$

\end{proof}

As a conclusion, we have $$\widehat{L}_{2,\fl}=0.$$

\subsection{Toward density of $\fl$-adic representations}
In this subsection, we aim at computing the natural density

$$L_2^{S,low}=\liminf_{x\rightarrow \infty}\frac{\#\{(g_1,g_2)\in B_2(x)\mid {\rm Im}\rho_{\phi^{(g_1,g_2)},\fl}\supset{\rm SL}_2(A_\fl) \textrm{ for }\fl\in S\}}{\#B_2(x)}.$$

\begin{thm}\label{lsqre}
Let $\phi^{(g_1,g_2)}$ be a rank-$2$ Drinfeld $A$-module such that ${\rm Im}\bar{\rho}_{\phi^{(g_1,g_2)},\fl}\supset\SL_2(\F_\fl).$ Suppose further that $g_2\not\in\fl$, then the image of mod $\fl^2$ Galois representation $\bar{\rho}_{\phi^{(g_1,g_2)},\fl^2}$ contains a non-scalar matrix which becomes identity after modulo $\fl$. 

\end{thm}

\begin{proof}

The proof is based on Lemma 4.2 and Proposition 4.4 of \cite{PR09}.

\end{proof}

\begin{cor}

$$\liminf_{x\rightarrow \infty}\frac{\#\{(g_1,g_2)\in B_2(x)\mid {\rm Im}\rho_{\phi^{(g_1,g_2)},\fl}\supset{\rm SL}_2(A_\fl)\}}{\#B_2(x)}\geqslant 1-\frac{1}{q^{\deg_T\fl}}$$

\end{cor}

\begin{proof}

By Theorem \ref{lsqre}, we have

$$\{(g_1,g_2)\in B_2(x)\mid {\rm Im}\rho_{\phi^{(g_1,g_2)},\fl}\supset{\rm SL}_2(A_\fl)\}\supset\{(g_1,g_2)\in B_2(x)\mid {\rm Im}\bar{\rho}_{\phi^{(g_1,g_2)},\fl}\supset{\rm SL}_2(A/\fl), \textrm{ and }g_2\not\in\fl\}.$$

Therefore,
$$
\begin{array}{ll}
\liminf_{x\rightarrow \infty}\frac{\#\{(g_1,g_2)\in B_2(x)\mid {\rm Im}\rho_{\phi^{(g_1,g_2)},\fl}\supset{\rm SL}_2(A_\fl)\}}{\#B_2(x)}\\
\ \\
\ \\
\geqslant\liminf_{x\rightarrow \infty}\frac{\#\{(g_1,g_2)\in B_2(x)\mid {\rm Im}\bar{\rho}_{\phi^{(g_1,g_2)},\fl}\supset{\rm SL}_2(A/\fl), \textrm{ and }g_2\not\in\fl\}}{\#B_2(x)}\\
\ \\
\ \\
\geqslant\liminf_{x\rightarrow \infty}\frac{\#\{(g_1,g_2)\in B_2(x)\mid {\rm Im}\bar{\rho}_{\phi^{(g_1,g_2)},\fl}\supset{\rm SL}_2(A/\fl)\}-\#\{(g_1,g_2)\in B_2(x)\mid g_2\not\in\fl\}}{\#B_2(x)}\\
\ \\
\ \\
=1-\limsup_{x\rightarrow \infty}\frac{\#\{(g_1,g_2)\in B_2(x)\mid g_2\not\in\fl\}}{\#B_2(x)}=1-\frac{1}{q^{\deg_T\fl}}
\end{array}
$$
\end{proof}

As an immediate result, we get

\begin{cor}\label{main2}
\begin{align*}
L_2^{S,low}\geqslant 1+\sum_{i=1}^{|S|}(-1)^i\cdot\left(\sum_{\fl_1,\cdots, \fl_i\in S \textrm{ distinct }}\frac{1}{q^{\deg_T\fl_1+\cdots+\fl_i}} \right).
\end{align*}
\end{cor}

On the other hand, notice that by Proposition \ref{weilpair} we have
$$
\begin{array}{cc}
\left\{(g_1,g_2)\in B_2(x)\mid {\rm Im}\rho_{\phi^{(g_1,g_2)},\fl}\not\subset{\rm GL}_2(A_\fl) \textrm{ for }\fl\in S\right\}\\
\ \\
\subset\left\{(g_1,g_2)\in B_2(x)\mid {\rm Im}\rho_{\phi^{(g_1,g_2)},\fl}\not\subset{\rm SL}_2(A_\fl) \textrm{ for }\fl\in S\right\}\cup\left\{(g_1,-\Delta)\in B_2(x)\mid {\rm Im}\rho_{\phi^{\Delta}}\neq\hat{A}^*\right\}.
\end{array}
$$
Hence Corollary \ref{cor1.1} and corollary \ref{main2} imply the following:

\begin{cor}\label{main3}
\begin{align*}
\liminf_{x\rightarrow \infty}\frac{\#\{(g_1,g_2)\in B_2(x)\mid {\rm Im}\rho_{\phi^{(g_1,g_2)},\fl}={\rm GL}_2(A_\fl) \textrm{ for }\fl\in S\}}{\#B_2(x)}\\
\ \\
\geqslant 1+\sum_{i=1}^{|S|}(-1)^i\cdot\left(\sum_{\fl_1,\cdots, \fl_i\in S \textrm{ distinct }}\frac{1}{q^{\deg_T\fl_1+\cdots+\fl_i}} \right)-\sum_{2\leqslant t\leqslant q-2, {t\mid q-1}}\frac{q-1}{(q^t-1)}
\end{align*}

In particular, the natural density of rank-$2$ Drinfeld modules with surjective $(T)$-adic Galois representation is bounded below by
\begin{align*}
\liminf_{x\rightarrow \infty}\frac{\#\{(g_1,g_2)\in B_2(x)\mid {\rm Im}\rho_{\phi^{(g_1,g_2)},(T)}={\rm GL}_2(A_{(T)})\}}{\#B_2(x)}\\
\ \\
\geqslant 1-\frac{1}{q}-\sum_{2\leqslant t\leqslant q-2, {t\mid q-1}}\frac{q-1}{(q^t-1)}=(q-1)\cdot (\frac{1}{q}-\sum_{2\leqslant t\leqslant q-2, {t\mid q-1}}\frac{1}{(q^t-1)})
\end{align*}

\end{cor}

\section*{Acknowledgement}

The author would like to thank Professor Fu-Tsun Wei for helpful and inspiring discussions to carry out this paper.

\bibliographystyle{alpha}
\bibliography{density_DM_rk2.bib}

\end{document}